\pdfoutput=1
\documentclass[Svgc,nospthms]{Svgc} 
\journalname{Graphs and Combinatorics}

\usepackage[utf8]{inputenc}

\usepackage{amssymb,amstext,amsmath,amsthm}
\usepackage{mathptmx}

\usepackage[numbers,sort&compress]{natbib}

\newcommand{\ellteekay}{\ensuremath{\ell{\mathsf T}k}\,}

\usepackage{graphicx}

\newcommand{\reffig}[1]{Figure \ref{fig.#1}}

\makeatletter
\newcommand{\Pr@}{\operatorname{Pr}}
\newcommand{\E@}{\operatorname{E}}
\newcommand{\Var@}{\operatorname{Var}}
\renewcommand{\Pr}[1]{\ensuremath{\Pr@\left[{#1}\right]}}
\newcommand{\E}[1]{\ensuremath{\E@\left[{#1}\right]}}
\newcommand{\Var}[1]{\ensuremath{\Var@\left[{#1}\right]}}
\makeatother

\newcommand{\card}[1]{\ensuremath{\left\vert #1 \right\vert}}

\usepackage{url}

\newtheorem{theorem}{Theorem}
\newtheorem{corollary}[theorem]{Corollary}
\newtheorem{lemma}[theorem]{Lemma}

\newtheorem{algorithm}[theorem]{Algorithm}

\newtheorem{problem}{Problem}

\newcommand{\reflem}[1]{Lemma \ref{lem.#1}}
\newcommand{\refthm}[1]{Theorem \ref{thm.#1}}
\newcommand{\refalg}[1]{Algorithm \ref{alg.#1}}
\newcommand{\refcor}[1]{Corollary \ref{cor.#1}}

\newcommand{\labellem}[1]{\label{lem.#1}}
\newcommand{\labelthm}[1]{\label{thm.#1}}
\newcommand{\labelalg}[1]{\label{alg.#1}}

\newcommand{\labelcor}[1]{\label{cor.#1}}

\usepackage[colorlinks=false]{hyperref}

\author{Ileana Streinu\inst{1}\thanks{Research of both authors funded by the NSF under grants NSF CCF-0430990 and NSF-DARPA CARGO CCR-0310661 to the first author.} \and Louis Theran\inst{2}}
\institute{Department of Computer Science, Smith College, Northampton, MA. \email{streinu@cs.smith.edu} 
\and Department of Computer Science, University of Massachusetts Amherst. \email{theran@cs.umass.edu}} 
\title{Sparsity-certifying Graph Decompositions} 

\newcommand{\restateenv}{ZZZ}
\newenvironment{restate}[1]{
  \renewcommand{\restateenv}{restate.#1}
  \newtheorem*{\restateenv}{\refthm{#1}}
  \begin{\restateenv}
}{\end{\restateenv}}

\newenvironment{restatecor}[1]{
  \renewcommand{\restateenv}{restate.#1}
  \newtheorem*{\restateenv}{\refcor{#1}}
  \begin{\restateenv}
}{\end{\restateenv}}

\usepackage{subfigure} 

\newcommand{\peb}{\ensuremath{\operatorname{peb}}} 
\newcommand{\grsp}{\ensuremath{\operatorname{span}}} 
\newcommand{\out}{\ensuremath{\operatorname{out}}} 

\begin{document}
	\maketitle
		\begin{abstract}
			We describe a new algorithm, the $(k,\ell)$-pebble game with colors, 
			and use it to obtain a characterization of the family of $(k,\ell)$-sparse graphs 
			and algorithmic solutions to a family of problems concerning tree decompositions of graphs. 
			Special instances of sparse graphs appear in rigidity theory and have received 
			increased attention in recent years. In particular, our colored pebbles generalize 
			and strengthen the previous results of Lee and Streinu \cite{pebblegame} and give a 
			new proof of the Tutte-Nash-Williams characterization of arboricity.  
			We also present a new decomposition that certifies sparsity based on the $(k,\ell)$-pebble 
			game with colors.  Our work also exposes connections between pebble game algorithms and 
			previous sparse graph algorithms by Gabow \cite{gabow:jcss-1995}, Gabow and 
			Westermann \cite{gabow:forests:1992} and Hendrickson
			 \cite{hendrickson:uniqueRealizability:1992}.
		\end{abstract}

		\section{Introduction and preliminaries}

		The focus of this paper
		is decompositions of $(k,\ell)$-sparse graphs into edge-disjoint 
		subgraphs that certify sparsity.  We use {\bf graph} to mean a multigraph, 
		possibly with loops.
		We say that a graph is {\bf $(k,\ell)$-sparse} if no subset of $n'$ vertices 
		spans more than $kn'-\ell$ edges in the graph; a $(k,\ell)$-sparse graph 
		with $kn'-\ell$ edges is {\bf $(k,\ell)$-tight}.
		We call 
		the range $k\le \ell\le 2k-1$ the upper range of sparse graphs
		and $0\le \ell\le k$ the lower range.

		In this paper, we  
		present  efficient algorithms for finding decompositions 
		that certify sparsity in the upper range of $\ell$.  Our algorithms
		also apply in the lower range, which was already addressed by 
		\cite{gabow:jcss-1995,gabow:forests:1992,Ed65,RoTa85,edmonds:matroidpolyhedra}.
	    A decomposition certifies 
		the sparsity of a graph if the sparse graphs and graphs admitting 
		the decomposition coincide.

		Our algorithms are based on a new characterization of sparse graphs, 
		which we call the {\bf pebble game with colors}.  
		The pebble game with colors is a 
		simple graph construction rule that produces a sparse graph along with a 
		sparsity-certifying decomposition.

		We define and study a canonical class of pebble game constructions, 
		which correspond to previously studied decompositions of sparse graphs 
		into edge disjoint trees.  
		Our results provide a unifying framework for all the previously known special cases, 
		including Nash-Williams-Tutte and \cite{whiteley:union-matroids,haas:2002}.
		Indeed, in the lower range,
		canonical pebble game constructions capture the properties of the augmenting paths 
		used in matroid union and intersection algorithms\cite{gabow:jcss-1995,gabow:forests:1992}.  
		Since the sparse
		graphs in the upper range are not known to be unions or intersections of the matroids 
		for which there are efficient augmenting path algorithms, these do not easily apply in
		the upper range.  Pebble game with colors constructions may thus be considered 
		a strengthening of augmenting paths to the upper range of matroidal sparse graphs.

\subsection{Sparse graphs}
	A graph is {\bf $(k,\ell)$-sparse} if for any non-empty subgraph with $m'$
	edges and $n'$ vertices,
	\(
		m' \le kn'-\ell. 
	\)
We observe that this condition implies that $0\le \ell\le 2k-1$, and from now
on in this paper we will make this assumption.  A sparse graph that has $n$ vertices and exactly $kn-\ell$ edges is called {\bf tight}.

For a graph $G=(V,E)$, and $V'\subset V$, we use the notation $\grsp (V')$ for the number of edges in the subgraph induced by $V'$.  In a directed graph, $\out (V')$ is the number of edges with the tail 
in $V'$ and the head in $V-V'$; for a subgraph induced by $V'$, we call such an edge an 
{\bf out-edge}.

There are two important types of subgraphs of sparse graphs. 
A {\bf block} is a tight subgraph of a sparse graph. A {\bf component} is a maximal block.

\begin{table}
	\begin{tabular}
		{|l|l|} \hline {\bf Term} & {\bf Meaning} \\
		\hline \hline Sparse graph $G$ & Every non-empty subgraph on $n'$ vertices has $\le kn'-\ell$ edges\\
		\hline Tight graph $G$ & $G=(V,E)$ is sparse and $\card{V}=n$, $\card{E}=kn-\ell$ \\
		\hline Block $H$ in $G$ & $G$ is sparse, and $H$ is a tight subgraph \\
		\hline Component $H$ of $G$ & $G$ is sparse and $H$ is a maximal block \\
		\hline Map-graph & Graph that admits an out-degree-exactly-one orientation \\
		\hline $(k,\ell)$-maps-and-trees &  Edge-disjoint union of $\ell$ trees and $(k-\ell)$ map-grpahs\\
		\hline \ellteekay & Union of $\ell$ trees, each vertex is in exactly $k$ of them \\
		\hline Set of tree-pieces of an $\ellteekay$ induced on $V'\subset V$ & Pieces of trees in the
		\ellteekay spanned by $E(V')$  \\
		\hline Proper \ellteekay & Every $V'\subset V$ contains $\ge\ell$ 
			pieces of trees from the \ellteekay \\
		\hline 
	\end{tabular}
	\caption{Sparse graph and decomposition terminology used in this paper.}	
	 \label{tab.sparse-terminology} 
\end{table}

Table \ref{tab.sparse-terminology} summarizes the sparse graph terminology used in this paper.

\subsection{Sparsity-certifying decompositions}
	A $k$-arborescence is a graph that admits a decomposition into $k$ edge-disjoint spanning trees. \reffig{colored-3-tree} shows an example of a 
	$3$-arborescence.  The $k$-arborescent graphs are described by the well-known theorems of Tutte \cite{tutte:decomposing-graph-in-factors-1961} and Nash-Williams \cite{nash-williams:decomposition-into-forests:1964} as exactly the $(k,k)$-tight graphs. 

	A {\bf map-graph} is a graph that admits an orientation such that the out-degree of each vertex is exactly one. A $k$-{\bf map-graph} is a graph that admits a decomposition into $k$ edge-disjoint map-graphs.  \reffig{colored-2-map} shows an example of a 2-map-graphs; the edges
	are oriented in one possible configuration certifying that each color forms a map-graph.
	 Map-graphs may be equivalently defined (see, e.g., \cite{oxley:matroid}) as having exactly one cycle per connected component.\footnote{Our terminology follows Lov{\'{a}}sz in \cite{lovasz:combinatorial-problems}. In the matroid literature map-graphs are sometimes known as bases of the bicycle matroid or spanning pseudoforests.}

	A {\bf $(k,\ell)$-maps-and-trees} is a graph that admits a decomposition into $k-\ell$ 
	edge-disjoint map-graphs and $\ell$ spanning trees.
	
Another characterization of map-graphs, which we will use extensively in this paper, is as the $(1,0)$-tight graphs \cite{whiteley:union-matroids,maps}. The $k$-map-graphs are evidently $(k,0)$-tight, and \cite{whiteley:union-matroids,maps} show that the converse holds as well. 

\begin{figure}[htbp]
	\centering
	\subfigure[]{\includegraphics[width=0.4\textwidth]{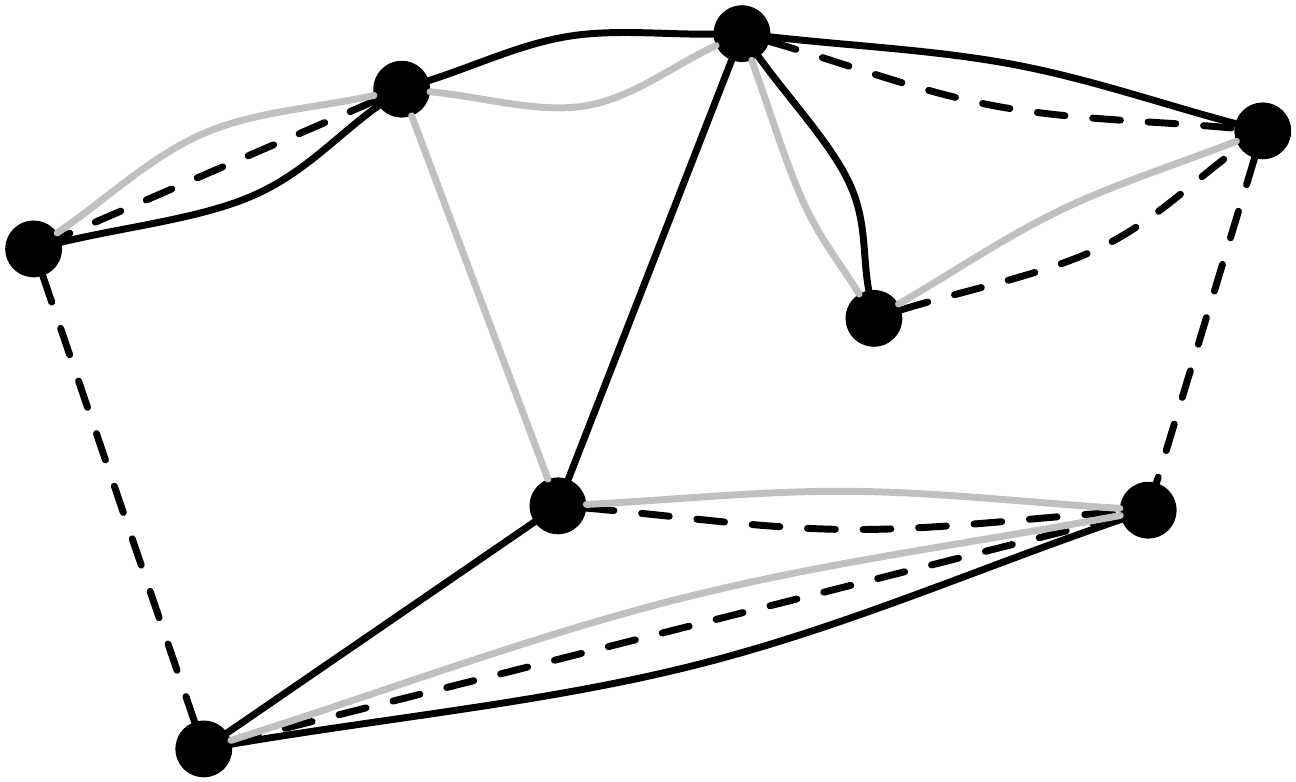}\label{fig.colored-3-tree}}
	\hspace{.3in}
	\subfigure[]{\includegraphics[width=0.33\textwidth]{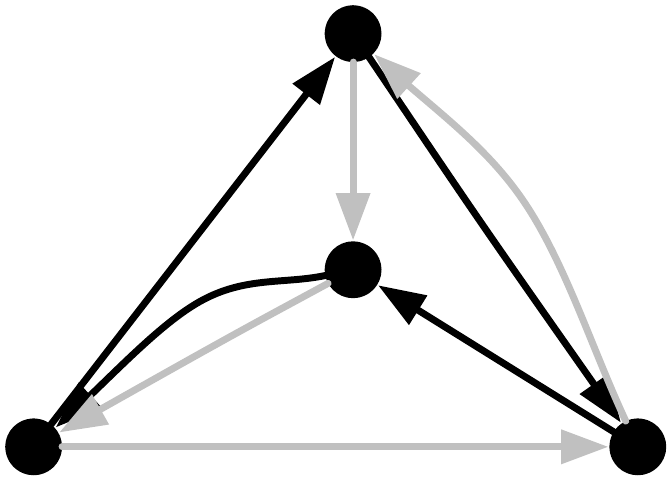}\label{fig.colored-2-map}}
	\hspace{.3 in}
	\subfigure[]{\includegraphics[width=0.27\textwidth]{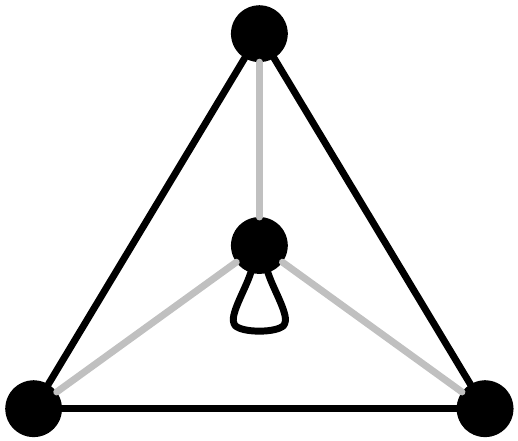}\label{fig.colored-2-1-m-a-t}}
	\caption{Examples of sparsity-certifying decompositions: (a) a $3$-arborescence; (b) a 2-map-graph;  (c)  a $(2,1)$-maps-and-trees.  Edges with the same line style belong to the 
	same subgraph.  The 2-map-graph is shown with a certifying orientation.}
	\end{figure}

A \ellteekay 
is a decomposition into $\ell$ edge-disjoint (not necessarily spanning) trees such that each vertex is in 
exactly $k$ of them.  \reffig{2-3-t-a-t} shows an example of a $3{\mathsf T}2$.

\begin{figure}[htbp]
	\centering
	\subfigure[]{\includegraphics[width=0.33\textwidth]{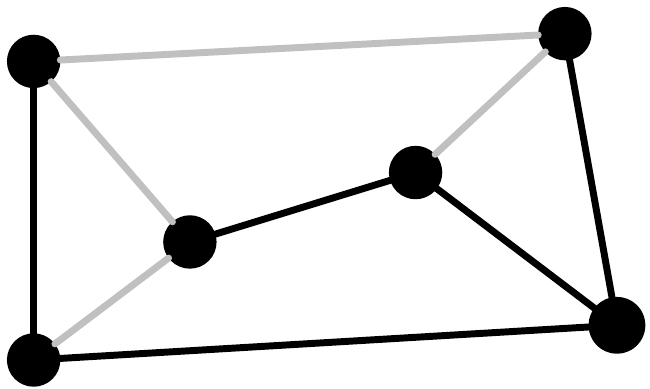}\label{fig.2-3-t-a-t}}
%	\hspace{.3in}
	\subfigure[]{\includegraphics[width=0.33\textwidth]{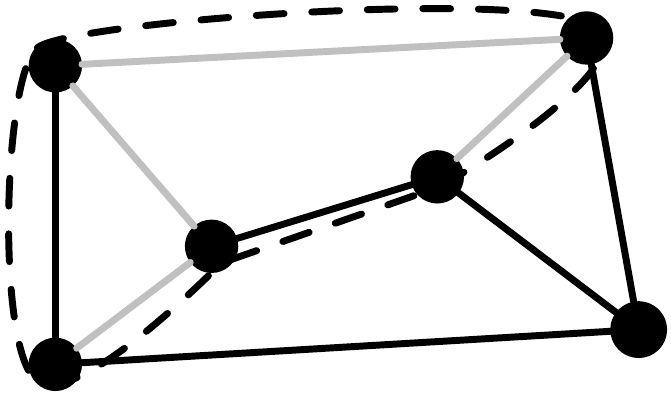}\label{fig.2-3-t-a-t-b}}
%	\hspace{.3in}
	\subfigure[]{\includegraphics[width=0.33\textwidth]{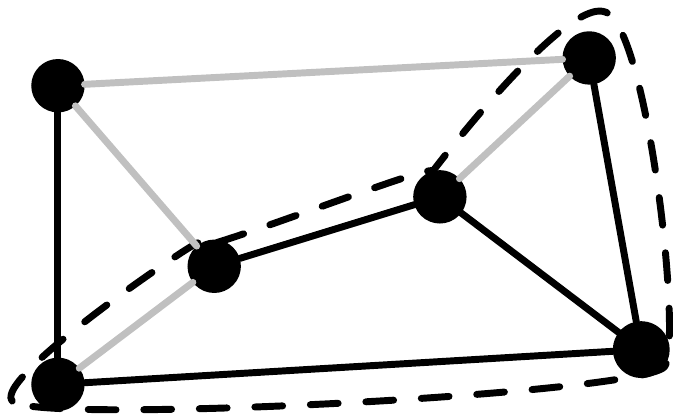}\label{fig.2-3-t-a-t-c}}
	\caption{(a) A graph with a $3\mathsf{T}2$ decomposition; one of the three trees is 
	a single vertex in the bottom right corner.  (b) The highlighted subgraph inside the dashed 
	countour has three black tree-pieces and one gray tree-piece.  (c) The highlighted subgraph inside 
	the dashed countour	has three gray tree-pieces (one is a single vertex) and one black tree-piece.}
	\end{figure}

	Given a subgraph $G'$ of a \ellteekay graph $G$, the {\bf set of tree-pieces} in $G'$ is 
	the collection of the components of the trees in $G$ induced by $G'$ (since $G'$
	is a subgraph each tree may contribute multiple pieces to the set of tree-pieces in $G'$). 
	We observe that these tree-pieces may come from the same tree or be single-vertex ``empty trees.'' 
	It is also helpful to note that the definition of a tree-piece is {\it relative to a specific 
	subgraph}.  An \ellteekay decomposition is \textbf{proper} if the set of tree-pieces in any 
	subgraph $G'$ has size at least $\ell$.
	
	\reffig{2-3-t-a-t} shows a graph with a $3\mathsf{T}2$ decomposition; we note that one
	of the trees is an isolated vertex in the bottom-right corner.  The subgraph in 
	\reffig{2-3-t-a-t-b} has three black tree-pieces and one gray tree-piece: an isolated vertex at
	the top-right corner, and two single edges.  These count as three tree-pieces, even though 
	they come from the same back tree when the whole graph in considered.  \reffig{2-3-t-a-t-c}
	shows another subgraph; in this case there are three gray tree-pieces and one black one.
	
	Table \ref{tab.sparse-terminology} contains the decomposition terminology used in this paper.
	
	\paragraph{The decomposition problem.} 
	We define the {\bf decomposition} problem for sparse graphs as taking a graph as
	its input and producing as output, a decomposition that can be used to certify 
	sparsity.  In this paper, we will study three kinds of outputs: maps-and-trees; 
	proper \ellteekay decompositions; and the pebble-game-with-colors decomposition, 
	which is defined in the next section.

\section{Historical background}
The well-known theorems of Tutte \cite{tutte:decomposing-graph-in-factors-1961} and Nash-Williams \cite{nash-williams:decomposition-into-forests:1964} relate 
the $(k,k)$-tight graphs to the existence of decompositions into edge-disjoint 
spanning trees.  Taking a matroidal viewpoint, Edmonds \cite{Ed65,edmonds:matroidpolyhedra} 
gave another proof of this result using matroid unions.  The equivalence of 
maps-and-trees graphs and tight graphs in the lower range is shown using matroid 
unions in \cite{whiteley:union-matroids}, and matroid augmenting paths are the 
basis of the algorithms for the lower range of \cite{RoTa85,gabow:jcss-1995,gabow:forests:1992}.

In rigidity theory a foundational theorem of Laman \cite{laman} shows that  
$(2,3)$-tight (Laman) graphs correspond to generically minimally rigid bar-and-joint 
frameworks in the plane.  Tay \cite{tay:rigidityMultigraphs-I:1984} proved an analogous  
result for body-bar frameworks in any dimension using $(k,k)$-tight graphs.  Rigidity by counts
motivated interest 
in the upper range, and Crapo \cite{crapo:rigidity:88} proved the equivalence of Laman
graphs and proper $3\mathsf{T}2$ graphs.  Tay \cite{Tay93} used this condition 
to give a direct proof of Laman's theorem and generalized the $3\mathsf{T}2$ condition
to all $\ell\mathsf{T}k$ for $k\le \ell\le 2k-1$.
Haas \cite{haas:2002} studied 
\ellteekay decompositions in detail and proved the equivalence of tight graphs and 
proper \ellteekay graphs for the general upper range.  We observe that aside 
from our new pebble-game-with-colors decomposition, all the combinatorial 
characterizations of the upper range of sparse graphs, including the counts,
have a geometric interpretation 
\cite{whiteley:union-matroids,laman,Tay93,tay:rigidityMultigraphs-I:1984}.

A pebble game algorithm was first proposed in \cite{jacobs:hendrickson:PebbleGame:1997a} as an elegant 
alternative to Hendrickson's Laman graph  algorithms
\cite{hendrickson:uniqueRealizability:1992}. Berg and Jordan  
\cite{berg:jordan:2003}, provided the formal analysis of the pebble game of \cite{jacobs:hendrickson:PebbleGame:1997a} and introduced 
the idea of playing the game on a directed graph. Lee and Streinu \cite{pebblegame} 
generalized the pebble game to the entire range of parameters $0\le \ell\le 2k-1$, 
and left as an open problem using the pebble game to find 
sparsity certifying decompositions.  

\section{The pebble game with colors}
Our {\bf pebble game with colors} is a set of rules for 
constructing graphs indexed by nonnegative integers $k$ and $\ell$.  
We will use the pebble game with colors as the basis of an efficient 
algorithm for the decomposition problem later in this paper.
Since the 
phrase ``with colors'' is necessary only for comparison to \cite{pebblegame},
we will omit it in the rest of the paper when the context is clear.

We now present the pebble game with colors.
The game is played by a single player on a fixed finite set of
vertices. The player makes a finite sequence of moves;  a move
consists in the addition and/or orientation of an edge. At any
moment of time, the state of the game is captured by a directed 
graph $H$, with colored pebbles on vertices and edges.  
The edges of $H$ are colored by the pebbles on them.  
While playing the pebble game all edges are directed, and we use the 
notation $vw$ to indicate a directed edge from $v$ to $w$.

We describe the pebble game with colors in terms of its initial configuration
and the allowed moves.

\begin{figure}[htbp]
	\centering
	\subfigure[]{\includegraphics[width=0.45\textwidth]{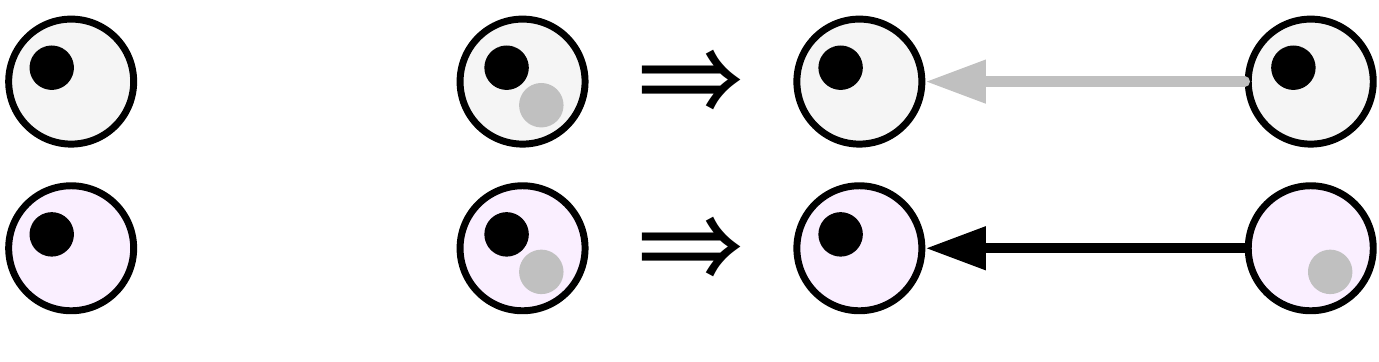}\label{fig.colored-add-edge}}
	\hspace{.3in}
	\subfigure[]{\includegraphics[width=0.45\textwidth]{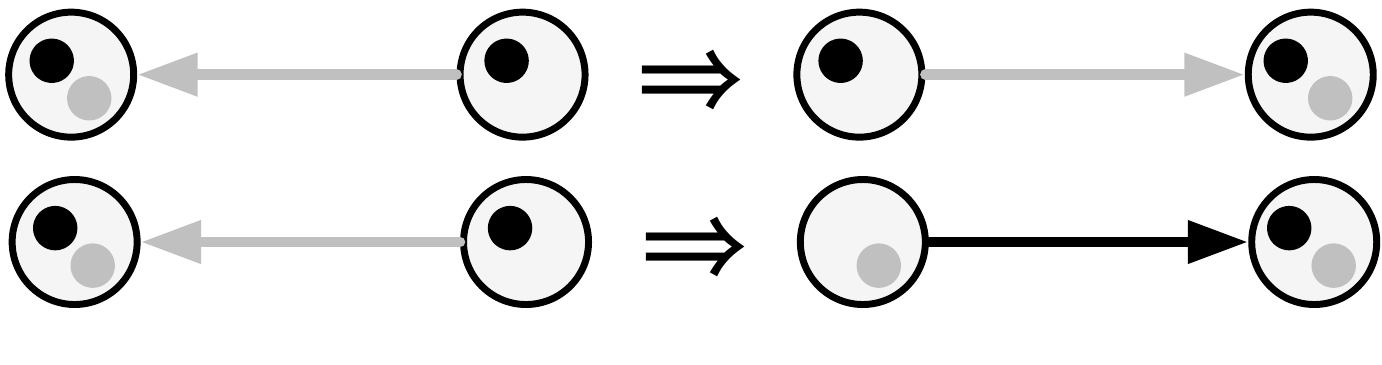}\label{fig.colored-pebble-slide}}
	\caption{Examples of pebble game with colors moves: (a) add-edge. (b) pebble-slide.
	Pebbles on vertices are shown as black or gray dots.  Edges are colored with the color
	of the pebble on them.}
\end{figure}

	\medskip
	{\bf Initialization:} In the beginning of the pebble game, $H$ has
	$n$ vertices and no edges.  We start by placing $k$ pebbles on each
	vertex of $H$, one of each color $c_i$, for $i=1,2,\ldots,k$.

	{\bf Add-edge-with-colors:} Let $v$ and $w$ be vertices with at least 
	$\ell+1$ pebbles on them. Assume (w.l.o.g.) that $v$ has at least one 
	pebble on it.
	Pick up a pebble from $v$, add the oriented edge $vw$ to $E(H)$ and put 
	the pebble picked up from $v$ on the new edge.

	\reffig{colored-add-edge} shows examples of the {\bf add-edge} move.

	{\bf Pebble-slide:} Let $w$ be a vertex with a pebble $p$ 
	on it, and let $vw$ be an edge in $H$. Replace $vw$ with $wv$ in $E(H)$; 
	put the pebble that was on  $vw$ on $v$;  and put $p$ on $wv$.
	
	Note that the color of an edge can change with a {\bf pebble-slide} move.
	 \reffig{colored-pebble-slide} shows examples.   The convention
	in these figures, and throughout this paper, is that pebbles on vertices
	are represented as colored dots, and that edges are shown in the color
	of the pebble on them.
	
	From the definition of the {\bf pebble-slide} move, it is easy to see that 
	a particular pebble is always either on the vertex where it started or on 
	an edge that has this vertex as the tail.  However, when making a sequence 
	of {\bf pebble-slide} moves that reverse the orientation of a path in $H$, 
	it is sometimes convenient to think of this path reversal sequence as 
	bringing a pebble from the end of the path to the beginning.

	\medskip
	The output of playing the pebble game is its complete configuration.

	{\bf Output:} At the end of the game, we obtain the directed graph $H$,
	along with the location and colors of the pebbles.  Observe that 
	since each edge has exactly one pebble on it, the pebble game configuration
	colors the edges.
	
	We say that the underlying undirected graph $G$ of $H$ is {\bf constructed}
	by the $(k,\ell)$-pebble game or that $H$ is a {\bf pebble-game graph}.

	Since each edge of $H$ has exactly one pebble on it, the pebble game's 
	configuration partitions the edges of $H$, and thus $G$, into $k$ different colors.  We call
	this decomposition of $H$ a {\bf pebble-game-with-colors decomposition}.  
	\reffig{k4-not-canonical-colors} 
	shows an example of a $(2,2)$-tight graph with a 
	pebble-game decomposition.  

	\begin{figure}[htbp]
		\centering
		\subfigure[]{
		\includegraphics[width=0.25\textwidth]{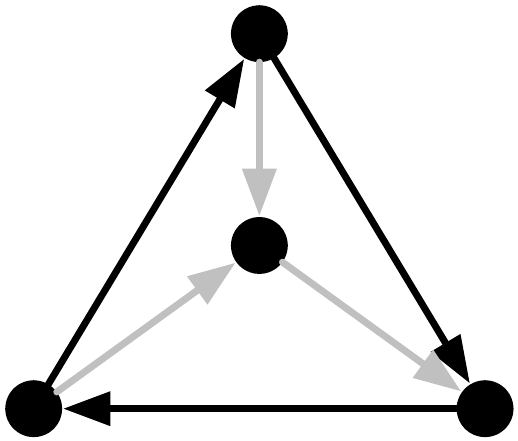}\label{fig.k4-not-canonical-colors}}
		\hspace{.3in}
		\subfigure[]{
		\includegraphics[width=0.25\textwidth]{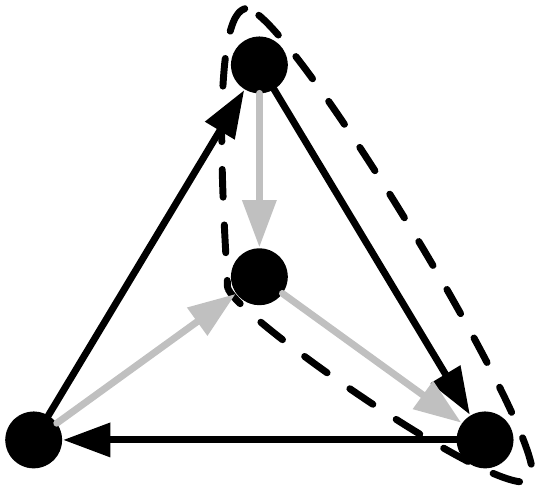}\label{fig.k4-not-canonical-colors-b}}
		\hspace{.3in}
		\subfigure[]{
		\includegraphics[width=0.25\textwidth]{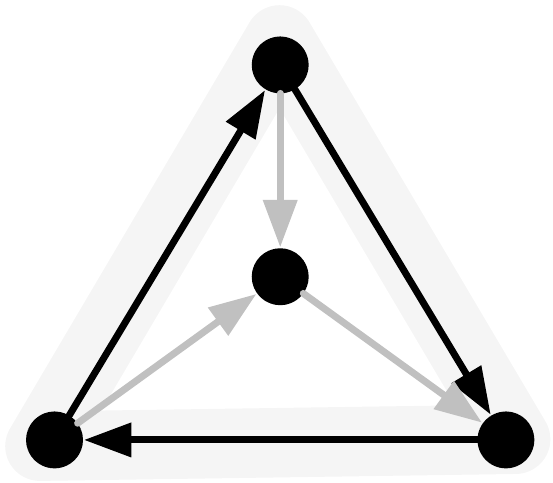}\label{fig.k4-not-canonical-colors-c}}
		\caption{A $(2,2)$-tight graph with one possible pebble-game decomposition.  The edges are oriented
		to show $(1,0)$-sparsity for each color.  (a) The graph $K_4$ with a pebble-game decomposition.
		There is an empty black tree at the center vertex and a gray spanning tree.  (b) The highlighted
		subgraph has two black trees and a gray tree; the black edges are part of a larger cycle but 
		contribute a tree to the subgraph.  (c) The highlighted subgraph (with a light gray background)
		has three empty gray trees;
		the black edges contain a cycle and do not contribute a piece of tree to the subgraph.}
	\end{figure}
	
	\medskip
	Let $G=(V,E)$ be pebble-game graph with the coloring induced by the 
	pebbles on the edges, and let $G'$ be a subgraph of $G$.  Then 
	the coloring of $G$ induces a set of monochromatic connected 
	subgraphs of $G'$ (there may be more than one of the same color).
	Such a monochromatic subgraph is called a \textbf{map-graph-piece}
	of $G'$ if it contains a cycle (in $G'$) and a \textbf{tree-piece}
	of $G'$ otherwise.  The \textbf{set of tree-pieces} of $G'$ 
	is the collection of tree-pieces induced by $G'$. As with the
	corresponding definition for $\ellteekay$s, the set of tree-pieces 
	is defined \emph{relative to a specific subgraph}; in particular 
	a tree-piece may be part of a larger cycle that includes edges 
	not spanned by $G'$.
	
	The properties of pebble-game decompositions are studied in Section \ref{pg-decomp},
	and \refthm{non-canonical-decomposition} shows that each color must be $(1,0)$-sparse.
	The orientation of the edges in \reffig{k4-not-canonical-colors} shows 
	this.
	
	For example \reffig{k4-not-canonical-colors} shows a $(2,2)$-tight graph with one 
	possible pebble-game decomposition.  The whole graph contains a gray tree-piece and a
	black tree-piece that is an isolated vertex.  The subgraph in \reffig{k4-not-canonical-colors-b}
	has a black tree and a gray tree, with the edges of the black tree coming from a cycle in the larger
	graph.  In \reffig{k4-not-canonical-colors-c}, however, the black cycle does not contribute a 
	tree-piece.  All three tree-pieces in this subgraph are single-vertex gray trees.

In the following discussion, we use the notation
$\peb(v)$ for the number of pebbles on $v$ and $\peb_i(v)$
to indicate the number of pebbles of colors $i$ on $v$.

Table \ref{tab.pebble-notations} lists the pebble game notation used in this paper.
\begin{table}
	\begin{tabular}
		{|l|l|} \hline {\bf Notation} & {\bf Meaning} \\
		\hline \hline
		$\grsp (V')$ & Number of edges spanned in $H$ by $V'\subset V$; i.e. $\card{E_{H}(V')}$\\
		\hline $\peb (V')$ &Number of pebbles on $V'\subset V$ \\
		\hline $\out (V')$ & Number of edges $vw$ in $H$ with $v\in V'$ and $w\in V-V'$ \\
		\hline $\peb_{i} (v)$ &Number of pebbles of color $c_{i}$ on $v\in V$ \\
		\hline $\out_{i} (v)$ & Number of edges $vw$ colored $c_i$ for $v\in V$ \\
		\hline
	\end{tabular}
	\caption{Pebble game notation used in this paper.} 
	\label{tab.pebble-notations} 
\end{table}

\section{Our Results} We describe our results in this section. 
The rest of the paper provides the proofs.

Our first result is a strengthening of the pebble games of \cite{pebblegame} 
to include colors.  It says that sparse graphs are exactly pebble game 
graphs.  Recall that from now on, all pebble games discussed in this 
paper are our pebble game with colors unless noted explicitly.
\begin{theorem}[{\bf Sparse graphs and pebble-game graphs coincide}]
\labelthm{sparse-graphs-are-pebble-graphs}
A graph $G$  is $(k,\ell)$-sparse  with $0\le\ell\le 2k-1$ if and only if $G$ is a pebble-game graph.
\end{theorem}
	
Next we consider pebble-game decompositions, showing that they are a generalization of 
proper \ellteekay decompositions that extend to the entire matroidal range of sparse
graphs.

\begin{theorem}[{\bf The pebble-game-with-colors decomposition}] \labelthm{non-canonical-decomposition}
	A graph $G$ is a pebble-game graph if and only if it admits a decomposition into $k$	
	edge-disjoint subgraphs such that each  is $(1,0)$-sparse and every subgraph of $G$ 
	contains at least $\ell$  tree-pieces of the $(1,0)$-sparse graphs in the decomposition.
\end{theorem}

The $(1,0)$-sparse subgraphs in the statement of \refthm{non-canonical-decomposition} 
are the colors of the pebbles; thus \refthm{non-canonical-decomposition}
gives a characterization of the pebble-game-with-colors decompositions 
obtained by playing the pebble game defined in the previous section.  
Notice the similarity between the 
requirement that the set of tree-pieces have size at least $\ell$ in 
\refthm{non-canonical-decomposition} and the definition of a proper $\ellteekay$.

Our next
results show that for {\it any} pebble-game graph, we can specialize its pebble game
construction to generate a decomposition that is a maps-and-trees or proper \ellteekay.
We call these specialized pebble game constructions {\bf canonical}, and using 
canonical pebble game constructions, we obtain new {\it direct} proofs of 
existing arboricity results.

We observe \refthm{non-canonical-decomposition}	that maps-and-trees 
 are special cases of the pebble-game decomposition: both spanning trees and 
spanning map-graphs are $(1,0)$-sparse, and each of the 
spanning trees contributes at least one piece of tree to every subgraph.

The case of proper \ellteekay graphs is more subtle; if each color in a 
pebble-game decomposition is a forest, then we have found a proper
\ellteekay, but this class is a subset of all possible proper \ellteekay
decompositions of a tight graph.  We show that this class of proper 
\ellteekay decompositions is sufficient to certify sparsity.

We now state the main theorem for the upper and lower range.

\begin{theorem}[{\bf Main Theorem (Lower Range): Maps-and-trees coincide with pebble-game graphs}]
	\labelthm{canonical-decomposition-I} Let $0\le \ell\le k$. A graph $G$ is a tight pebble-game graph if and only if $G$ is a $(k,\ell)$-maps-and-trees. 
\end{theorem}

\begin{theorem}[{\bf Main Theorem (Upper Range): Proper \ellteekay graphs coincide with pebble-game graphs}]
	Let $k\le \ell\le 2k-1$. A graph $G$ is a tight pebble-game graph if and only if it is a proper \ellteekay with $kn-\ell$ edges. \labelthm{canonical-decomposition-II} 
\end{theorem}
	
As corollaries, we obtain the existing decomposition results for sparse graphs.
\begin{corollary}
	[\textbf{Nash-Williams \cite{nash-williams:decomposition-into-forests:1964}, Tutte \cite{tutte:decomposing-graph-in-factors-1961}, White and Whiteley \cite{whiteley:union-matroids}}] 
		\labelcor{m-a-t-equals-tight}
	Let $\ell\le k$. A graph $G$ is tight if and only if has a $(k,\ell)$-maps-and-trees decomposition. 
\end{corollary}

\begin{corollary}
	[\textbf{Crapo \cite{crapo:rigidity:88}, Haas \cite{haas:2002}}] 
	\labelcor{t-a-t-equals-tight}
	Let $k\le \ell\le 2k-1$. A graph $G$ is tight if and only if it is a proper \ellteekay. 
\end{corollary}

\paragraph{Efficiently finding canonical pebble game constructions.}
The proofs of \refthm{canonical-decomposition-I} and \refthm{canonical-decomposition-II}
lead to an obvious algorithm with $O(n^3)$ running time for the {\bf decomposition}
problem.  Our last result improves on this, showing that a canonical pebble game 
construction, and thus a maps-and-trees or proper \ellteekay decomposition can be 
found using a pebble game algorithm in $O(n^2)$ time and space.

These time and space bounds mean that our algorithm can be 
combined with those of \cite{pebblegame} without any change in complexity.

\section{Pebble game  graphs} 
In this section we prove \refthm{sparse-graphs-are-pebble-graphs},
 a strengthening of results from \cite{pebblegame} to the pebble game with colors.
Since many of the relevant properties of the pebble game with colors carry over
directly from the pebble games of \cite{pebblegame}, we refer the reader there for the proofs.

We begin by establishing some invariants that hold during the execution of the pebble game.
	\begin{lemma}[{\bf Pebble game invariants}]
		During the execution of the pebble game, the following invariants are maintained in $H$: 
		\begin{enumerate}
			\item[{\bf (I1)}] There are at least $\ell$ pebbles on $V$.  \cite{pebblegame}
			\item[{\bf (I2)}] For each vertex $v$, $\grsp (v) + \out (v) + \peb (v)=k$.  \cite{pebblegame}
			\item[{\bf (I3)}] For each $V'\subset V$, $\grsp (V')+\out (V')+\peb (V')=kn'$. \cite{pebblegame}
			\item[{\bf (I4)}] For every vertex $v\in V$, $\out_i (v)+\peb_i (v)=1$. 
			\item[{\bf (I5)}] Every maximal path consisting only of edges with color $c_i$ ends in either the first vertex with a pebble of color $c_i$ or a cycle. 
		\end{enumerate}
		\labellem{pebble-game-invariants} 
\end{lemma}

	\begin{proof}
		{\bf (I1)}, 
		{\bf (I2)}, 
		and {\bf (I3)} come directly from \cite{pebblegame}.
		
		{\bf (I4)} This invariant clearly holds at the initialization phase of the pebble game with colors. That {\bf add-edge} and {\bf pebble-slide} moves preserve {\bf (I4)} is clear from inspection. 
		
		{\bf (I5)} By {\bf (I4)}, a monochromatic path of edges is forced to end only at a vertex with a pebble of the same color on it. If there is no pebble of that color reachable, then the path must eventually visit some vertex twice. 
	 \end{proof}
	
	From these invariants, we can show that the pebble game constructible graphs are sparse.
\begin{lemma}[{\bf Pebble-game graphs are sparse \cite{pebblegame}}]\labellem{pebble-graphs-are-sparse}
		Let $H$ be a graph constructed with the pebble game. Then $H$ is sparse. 
		If there are exactly $\ell$ pebbles on $V(H)$, then $H$ is tight.  
\end{lemma}
	
The main step in proving that every sparse graph is a pebble-game graph is the following.  Recall
that by bringing a pebble to $v$ we mean reorienting $H$ with {\bf pebble-slide} moves to reduce the 
out degree of $v$ by one.
\begin{lemma}[{\bf The $\ell+1$ pebble condition} \cite{pebblegame}]\labellem{can-bring-another-pebble}
Let $vw$ be an edge such that $H+vw$ is sparse. If $\peb (\{v,w\})<\ell+1$, 
then a pebble not on $\{v,w\}$ can be brought to either $v$ or $w$.  
\end{lemma}

It  follows that any sparse graph has a pebble game construction. 
\begin{restate}{sparse-graphs-are-pebble-graphs}[{\bf Sparse graphs and pebble-game graphs coincide}]
A graph $G$  is $(k,\ell)$-sparse  with $0\le\ell\le 2k-1$ if and only if $G$ is a pebble-game graph.
\end{restate}
	
\section{The pebble-game-with-colors decomposition}\label{pg-decomp}
In this section we prove \refthm{non-canonical-decomposition}, which characterizes
all pebble-game decompositions.
We start with the following lemmas about the structure of monochromatic connected components in $H$,
the directed graph maintained during the pebble game.
	
\begin{lemma}[{\bf Monochromatic pebble game subgraphs are $(1,0)$-sparse}]\labellem{each-color-is-map-sparse}
	Let $H_i$ be the subgraph of $H$ induced by edges with pebbles of color $c_i$ on them. 
	Then $H_i$ is $(1,0)$-sparse, for $i=1,\ldots,k$.  
\end{lemma}
\begin{proof}
	By {\bf (I4)} $H_i$ is a set of edges with out degree at most one for every vertex. 
	 
\end{proof}

\begin{lemma}[{\bf Tree-pieces in a pebble-game graph}]
\labellem{subtrees}
Every subgraph of the directed graph $H$ in a pebble 
game construction contains at least $\ell$ monochromatic tree-pieces, 
and each of these is rooted at either a vertex with a pebble on it  
or a vertex that is the tail of an out-edge.
\end{lemma}
Recall that an out-edge from a subgraph $H'=(V',E')$ is an edge $vw$ with $v\in V'$
and $vw\notin E'$.
\begin{proof}
Let $H'=(V',E')$ be a non-empty subgraph of $H$, and assume without loss of generality that $H'$ is 
induced by $V'$.  By {\bf (I3)}, $\out (V')+\peb (V')\ge \ell$.  We will show that each pebble and 
out-edge tail is the root of a tree-piece.

Consider a vertex $v\in V'$ and a color $c_i$.  By {\bf (I4)} there is a unique monochromatic 
directed path of color $c_i$ starting at $v$.  By {\bf (I5)}, if this path ends at a pebble,
it does not have a cycle.  Similarly, if this path reaches a vertex that is the tail of an out-edge
also in color $c_i$ (i.e., if the monochromatic path from $v$ leaves $V'$), then the path cannot 
have a cycle in $H'$.

Since this argument works for any vertex in any color, for each color there is a partitioning 
of the vertices into those that can reach each pebble, out-edge tail, or cycle.  It follows
that each pebble and out-edge tail is the root of a monochromatic tree, as desired. 
 \end{proof}

Applied to the whole graph \reflem{subtrees} gives us the following.

\begin{lemma}[{\bf Pebbles are the roots of trees}]\labellem{roots}
In any pebble game configuration, each pebble of color $c_i$ is the 
root of a (possibly empty) monochromatic tree-piece of color $c_i$.
\end{lemma}

{\bf Remark:}  Haas showed in \cite{haas:2002} that in a \ellteekay, a 
subgraph induced by $n'\ge 2$ vertices with $m'$ edges has exactly $kn'-m'$
tree-pieces in it.  \reflem{subtrees} strengthens Haas' result by extending it 
to the lower range and giving a construction that finds the tree-pieces, showing 
the connection between the $\ell+1$ pebble condition and the hereditary 
condition on proper \ellteekay. 

We conclude our investigation of arbitrary pebble game constructions 
with a description of the decomposition induced by the pebble game with colors.
	
\begin{restate}{non-canonical-decomposition}[{\bf The pebble-game-with-colors decomposition}] 
		A graph $G$ is a pebble-game graph if and only if it admits a decomposition into $k$	
		edge-disjoint subgraphs such that each  is $(1,0)$-sparse and every subgraph of $G$ 
		contains at least $\ell$  tree-pieces of the $(1,0)$-sparse graphs in the decomposition.
\end{restate}
	
\begin{proof}
Let $G$ be a pebble-game graph.  The existence of the $k$ edge-disjoint
$(1,0)$-sparse subgraphs was shown in  \reflem{each-color-is-map-sparse}, and
\reflem{subtrees} proves the condition on subgraphs.

For the other direction, we observe that a color 
$c_i$ with $t_i$ tree-pieces in a given subgraph can span at most $n-t_i$
edges; summing over all the colors shows that a graph with a pebble-game decomposition
must be sparse.  Apply \refthm{sparse-graphs-are-pebble-graphs} to complete the proof.
 \end{proof}
	
{\bf Remark: } We observe that a pebble-game decomposition for a 
Laman graph may be 
read out of the bipartite matching used in Hendrickson's Laman graph 
extraction algorithm \cite{hendrickson:uniqueRealizability:1992}.  
Indeed, pebble game orientations have a natural 
correspondence with the bipartite matchings used in 
\cite{hendrickson:uniqueRealizability:1992}.

Maps-and-trees are a special case of pebble-game decompositions for tight graphs:
if there are no cycles in $\ell$ of the colors, then the trees rooted at the 
corresponding $\ell$ pebbles must be spanning, since they have $n-1$ edges.
Also, if each color forms a forest in an upper range pebble-game decomposition, then
the tree-pieces condition ensures that the pebble-game decomposition is a proper
$\ellteekay$.

In the next section, we show that the pebble game can be specialized to correspond
to maps-and-trees and proper \ellteekay decompositions.

\section{Canonical Pebble Game Constructions}
In this section we prove the main theorems  (\refthm{canonical-decomposition-I} and \refthm{canonical-decomposition-II}), continuing the investigation of decompositions 
induced by pebble game constructions by studying the 
case where a minimum number of monochromatic cycles are created.  
The main idea, captured in \reflem{kill-m2-moves-locally} and illustrated in 
\reffig{m2-meta-picture}, is to avoid creating cycles while collecting pebbles.  We 
show that this is always possible, implying that monochromatic map-graphs are 
created only when we add more than $k(n'-1)$ edges to some set of $n'$ vertices.  For the lower
range, this implies that every color is a forest.
Every decomposition characterization of tight graphs discussed above follows immediately from
the main theorem, giving new proofs of the previous results in a unified framework.

In the proof, we will use two specializations of the pebble game moves.  
The first is a modification of the {\bf add-edge} move.

{\bf Canonical add-edge:} When performing an {\bf add-edge} move, cover the new edge with 
a color that is on both vertices if possible.  If not, then take the highest 
numbered color present.

The second is a restriction on which {\bf pebble-slide} moves we allow.

{\bf Canonical pebble-slide:} A {\bf pebble-slide} move is allowed only when it does
not create a monochromatic cycle.

We call a pebble game construction that uses only these moves {\bf canonical}.  In this section
we will show that every pebble-game graph has a canonical pebble game construction 
(\reflem{can-kill-m1-moves} and \reflem{kill-m2-moves-locally}) and that canonical 
pebble game constructions correspond to proper \ellteekay and maps-and-trees decompositions
(\refthm{canonical-decomposition-I} and \refthm{canonical-decomposition-II}).

We begin with a technical lemma that motivates the definition of canonical 
pebble game constructions.  It shows that the situations disallowed by 
the canonical moves are {\it all} the ways for cycles to form in 
the lowest $\ell$ colors.

\begin{lemma}[{\bf Monochromatic cycle creation}]\labellem{how-maps-form}
	Let $v\in V$ have a pebble $p$ of color $c_{i}$ on it and let $w$
	be a vertex in the same tree of color $c_i$ as $v$.	A  monochromatic cycle colored $c_{i}$ is 
	created in exactly one of the following ways: 
	\begin{enumerate}
		\item[{\bf (M1)}] The edge $vw$ is added with an {\bf add-edge} move.
		\item[{\bf (M2)}] The edge $wv$ is reversed by a {\bf pebble-slide} move and 
		the pebble $p$ is used to cover the reverse edge $vw$.	
	\end{enumerate}
	\end{lemma}
	 \begin{proof}
	 	Observe that the preconditions in the statement of the lemma are implied by
	 	 \reflem{pebble-game-invariants}. By \reflem{roots}  monochromatic 
	 	cycles form when the last pebble of color $c_{i}$ is removed from a 
	 	connected monochromatic subgraph. {\bf (M1)} and {\bf (M2)} 
	 	are the only ways to do this in a pebble game construction, since the 
		color of an edge only changes when it is inserted the first time 
		or a new pebble is put on it by a {\bf pebble-slide} move.
	  \end{proof}
	
	\reffig{m1-create-map} and \reffig{m2-create-map} show examples of 
	{\bf (M1)} and {\bf (M2)} map-graph creation moves, respectively, 
	in a $(2,0)$-pebble game construction.
	
	\begin{figure}[htbp]
		\centering
		\subfigure[]{\includegraphics[width=0.6\textwidth]{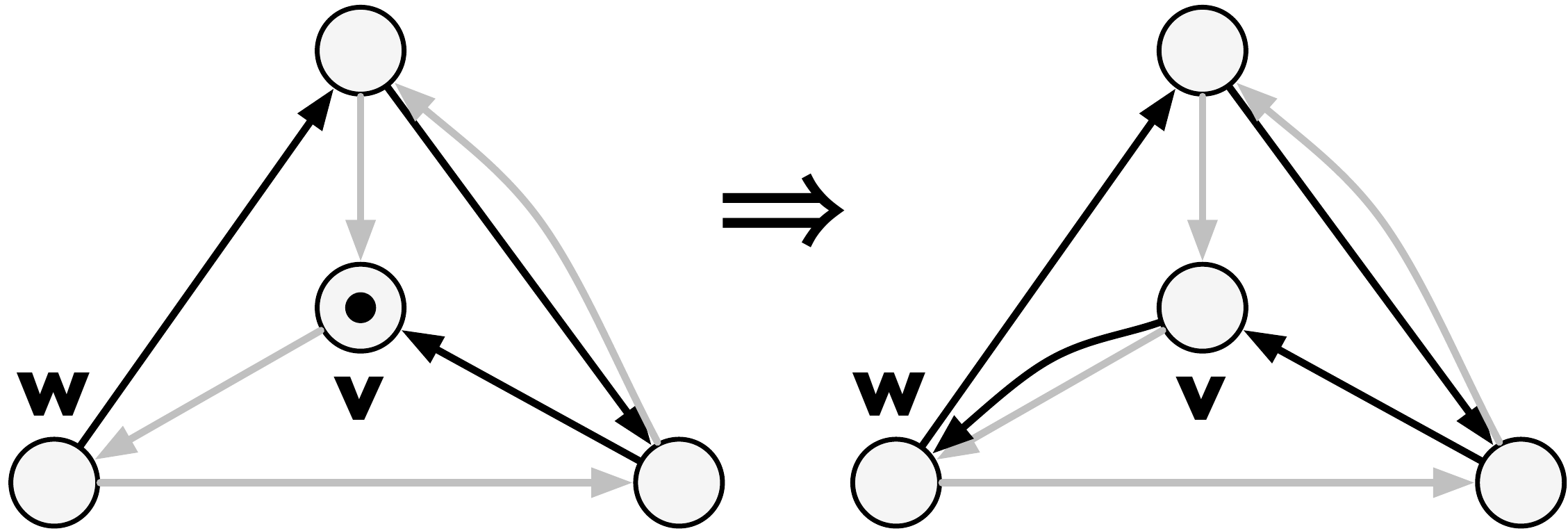}\label{fig.m1-create-map}}
%		\hspace{.1in}
		\subfigure[]{\includegraphics[width=0.6\textwidth]{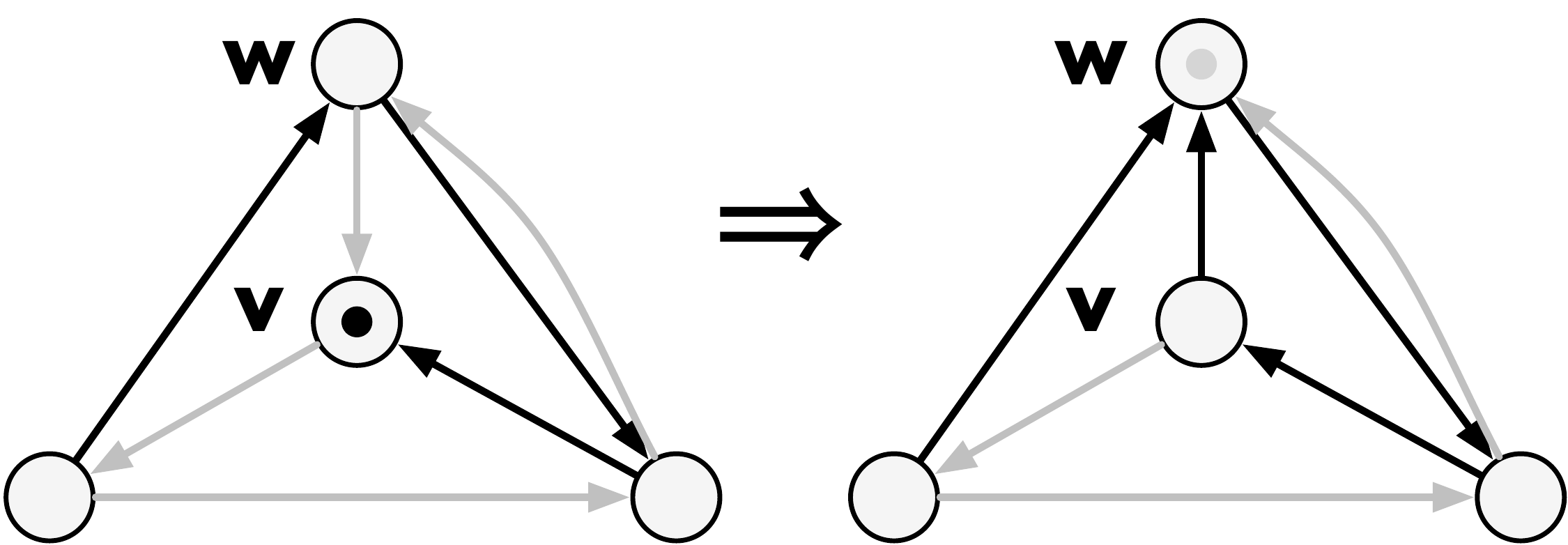}\label{fig.m2-create-map}}
		\caption{Creating monochromatic cycles in a $(2,0)$-pebble game. (a) A type {\bf (M1)} move creates a cycle by adding a black edge.  (b) A type {\bf (M2)} move creates a cycle with a {\bf pebble-slide} move.
		The vertices are labeled according to their role in the definition of the moves.}
		\end{figure}
	
	We next show that if a graph has a pebble game construction, then it has a canonical pebble 
	game construction.  This is done in two steps, considering the cases {\bf (M1)} and {\bf (M2)}
	separately.  The proof gives two constructions that implement the {\bf canonical add-edge}
	and {\bf canonical pebble-slide} moves.
	
	\begin{lemma}[{\bf The canonical add-edge move}]
		Let $G$ be a graph with a pebble game construction. 
		Cycle creation steps of type {\bf (M1)} can be eliminated in colors $c_{i}$ for $1\le i\le \ell'$, 
		where $\ell'=\min\{k,\ell\}$. \labellem{can-kill-m1-moves} 
	\end{lemma}
	\begin{proof}
		For {\bf add-edge} moves, cover the edge with a color present on both $v$ and $w$ 
		if possible. If this is not possible, then there are $\ell+1$ 
		distinct colors present. Use the highest numbered color to cover 
		the new edge. 
	 \end{proof}
	
	{\bf Remark:} We note that in the upper range, there is always a repeated color, so 
	{\it no} {\bf canonical add-edge} moves create cycles in the upper range.

	The {\bf canonical pebble-slide} move is defined by a global condition.  To prove that
	we obtain the same class of graphs using only {\bf canonical pebble-slide} moves, 
	we need to extend \reflem{can-bring-another-pebble} to only 
	canonical moves.  The main step is to show that if there is {\it any }
	sequence of moves that reorients a path from $v$ to $w$, then there is 
	a sequence of canonical moves that does the same thing.
	
	\begin{lemma}[{\bf The canonical pebble-slide move}]\labellem{kill-m2-moves-locally}
		Any sequence of {\bf pebble-slide} moves leading to an {\bf add-edge}
		move can be replaced with one that has no {\bf (M2)} steps and allows the
		same {\bf add-edge} move. 
	\end{lemma}
	In other words, if it is possible to collect $\ell+1$ pebbles on the ends 
	of an edge to be added, then it is possible to do this without creating 
	any monochromatic cycles.
		
	\reffig{m2-move-eliminate} and \reffig{m2-move-eliminate-2} illustrate the
	construction used in the proof of \reflem{kill-m2-moves-locally}.  We call this 
	the {\bf shortcut construction} by analogy to matroid union and intersection 
	augmenting paths used in previous work on the lower range.
	
	\reffig{m2-meta-picture} shows the structure of the proof.  The shortcut construction 
	removes an {\bf (M2)} step at the beginning of a sequence that 
	reorients a path from $v$ to $w$ with pebble-slides.  Since one application of the shortcut construction 
	reorients a simple path from a vertex $w'$ to $w$, and a path from $v$ to $w'$ is 
	preserved, the shortcut construction can be applied inductively to find the sequence of moves
	we want.
	
	\begin{figure}[htbp]
	\centering
	\subfigure[]{ \includegraphics[width=0.33\textwidth]{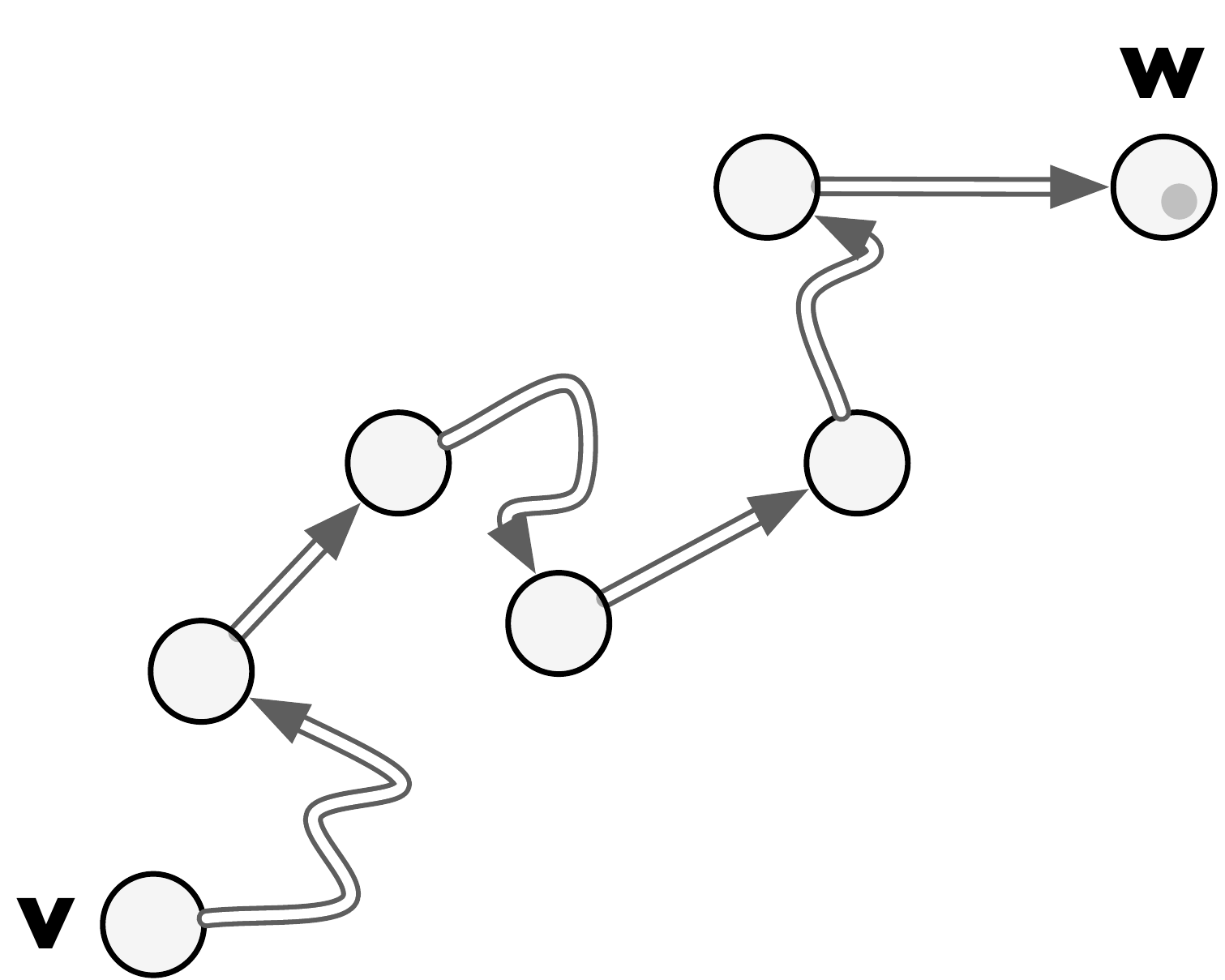} }
	\hspace{.3in}
	\subfigure[]{ \includegraphics[width=0.33\textwidth]{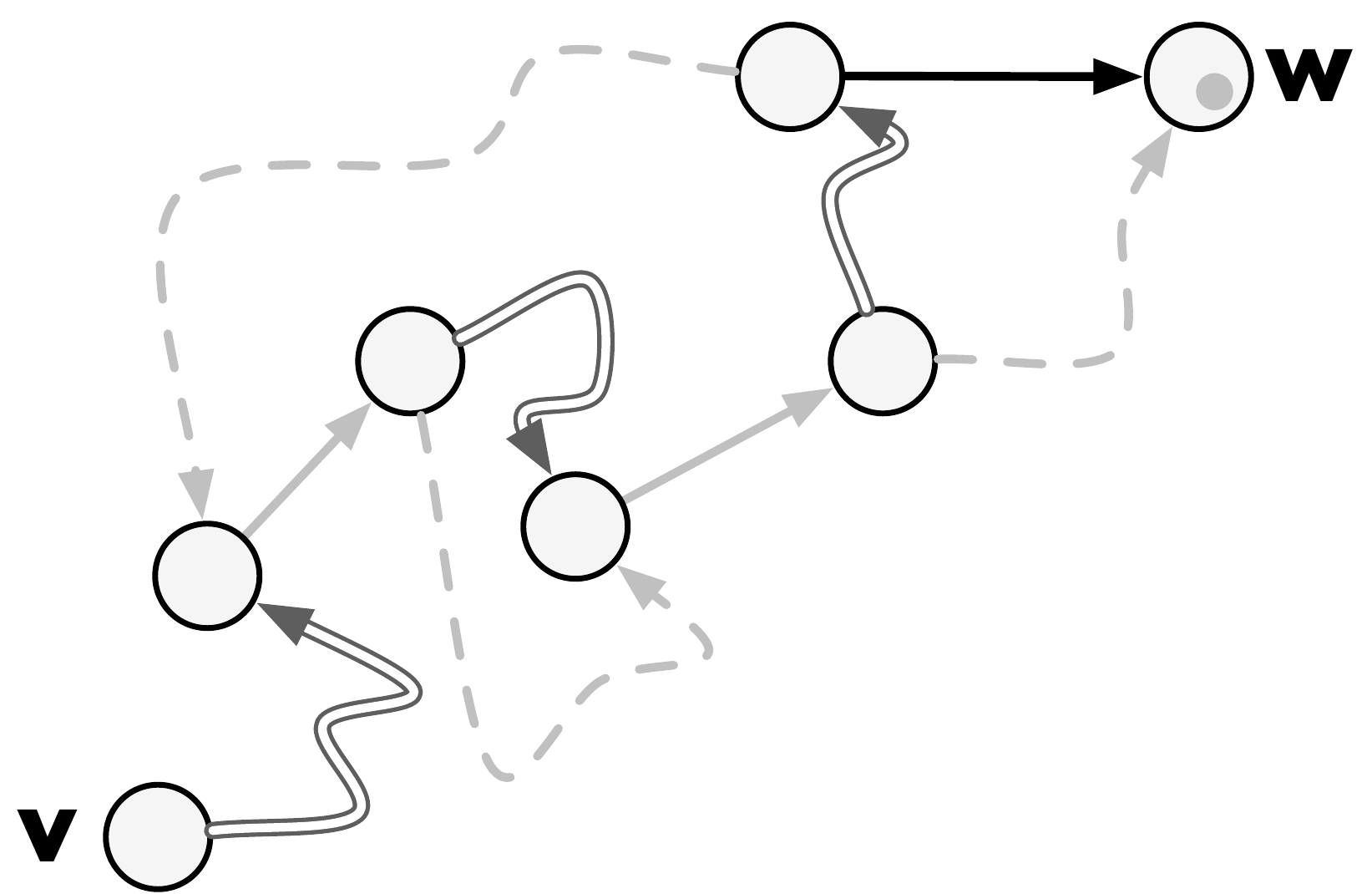} }
	\hspace{.3in}
	\subfigure[]{ \includegraphics[width=0.33\textwidth]{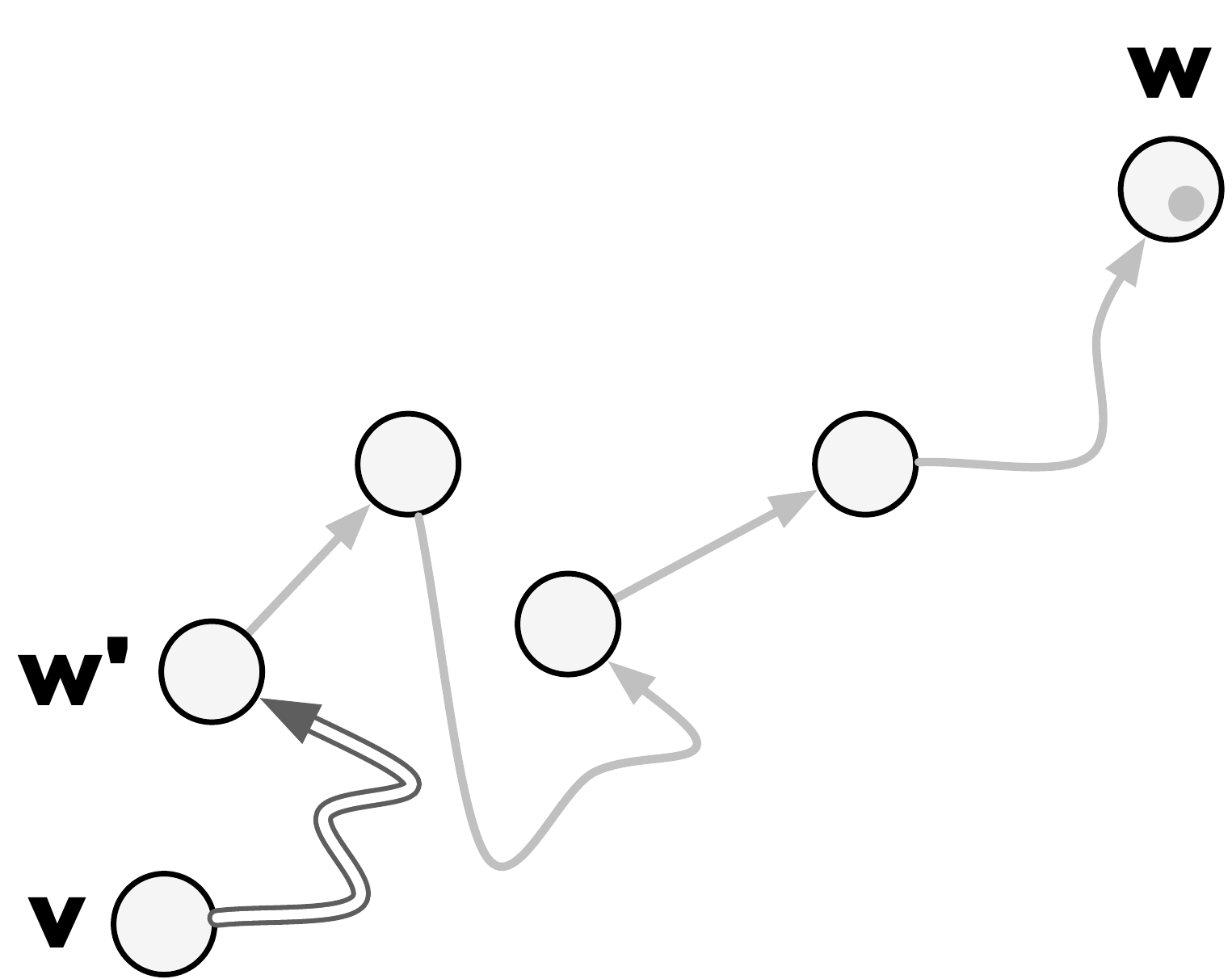} }
	
	\caption{Outline of the shortcut construction: (a) An arbitrary simple path from $v$ to $w$ with
	curved lines indicating simple paths.
	(b) An {\bf (M2)} step. The black edge, about to be flipped, would create a cycle, 
	shown in dashed and solid gray, of the (unique) gray tree rooted at $w$.  The solid gray edges were part of the original path from (a).
	(c) The shortened path
	to the gray pebble; the new path follows the gray tree all the way from the first time the original
	path touched the gray tree at $w'$.  The path from $v$ to $w'$ is simple, and the shortcut construction
	can be applied inductively to it.}
	\label{fig.m2-meta-picture}
	\end{figure}
	
	\begin{figure}[htbp]
	\centering
	\subfigure[]{ \includegraphics[width=0.6\textwidth]{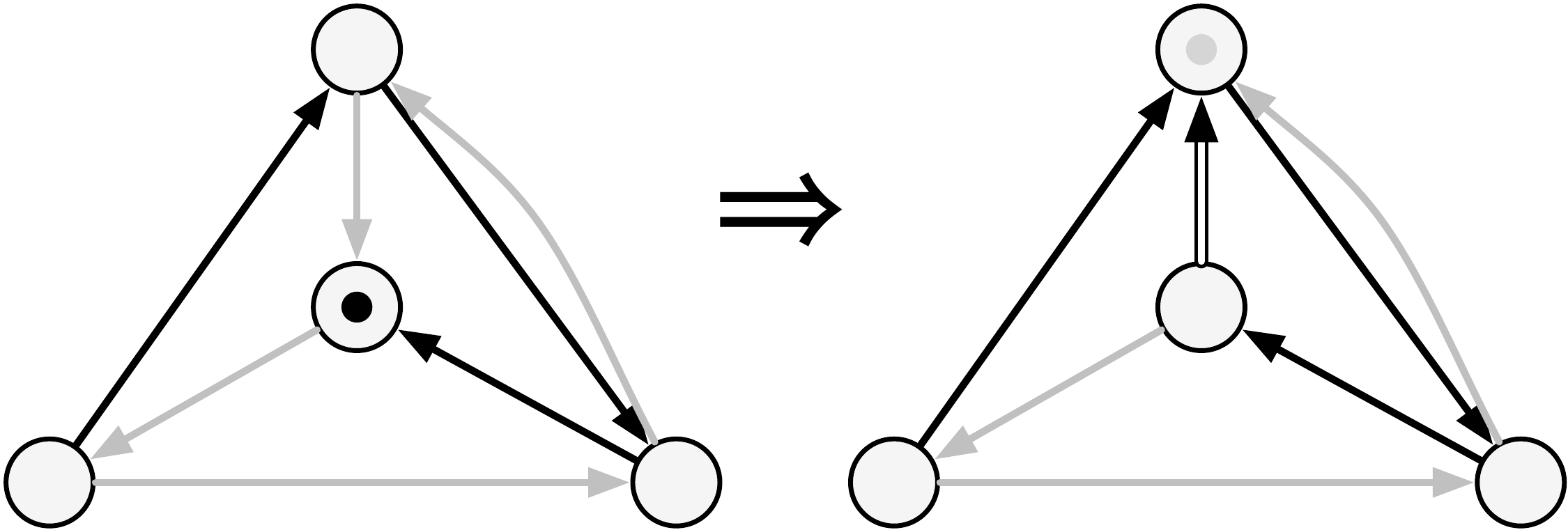} }
%	\hspace{.3in}
	\subfigure[]{ \includegraphics[width=0.6\textwidth]{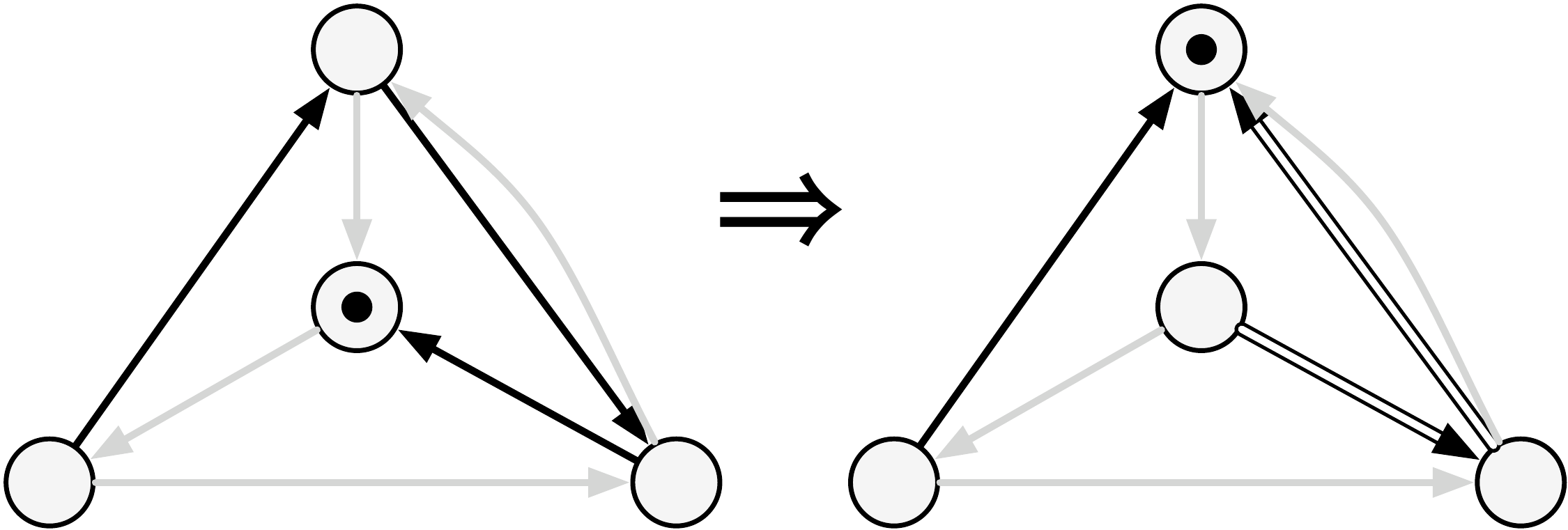} }
	
	\caption{Eliminating {\bf (M2)} moves: (a) an {\bf (M2)} move; (b) avoiding the 
	{\bf (M2)} by moving along another path.  The path where the pebbles move is indicated
	by doubled lines.}
	\label{fig.m2-move-eliminate}
	\end{figure}
	
	\begin{figure}[htbp]
	\centering
	\subfigure[]{ \includegraphics[width=0.6\textwidth]{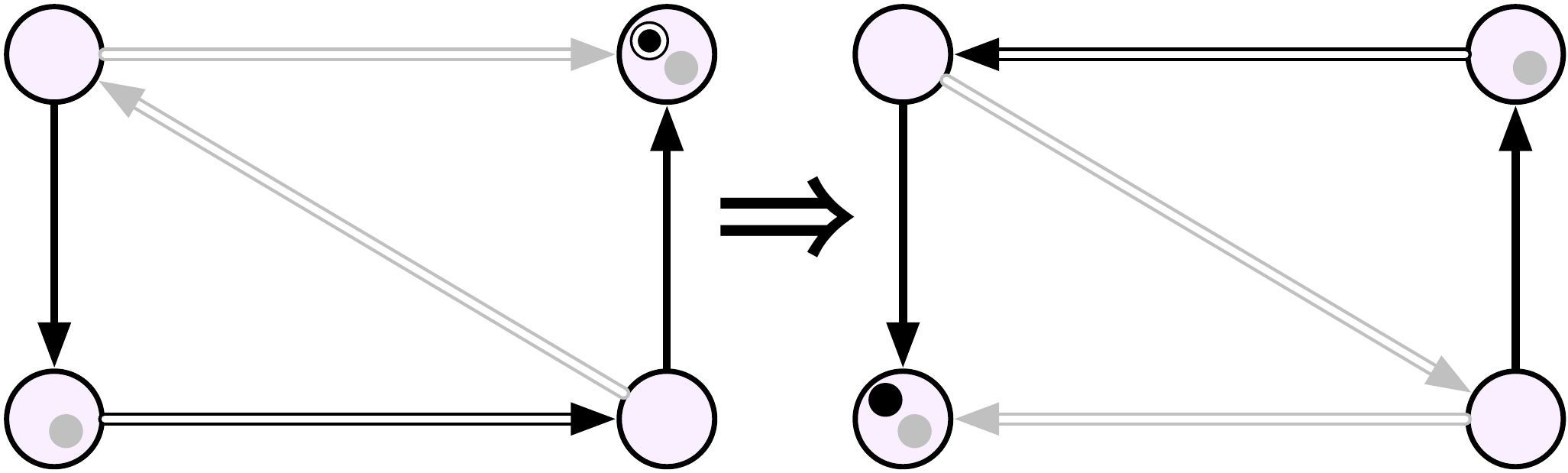} }
%	\hspace{.3in}
	\subfigure[]{ \includegraphics[width=0.6\textwidth]{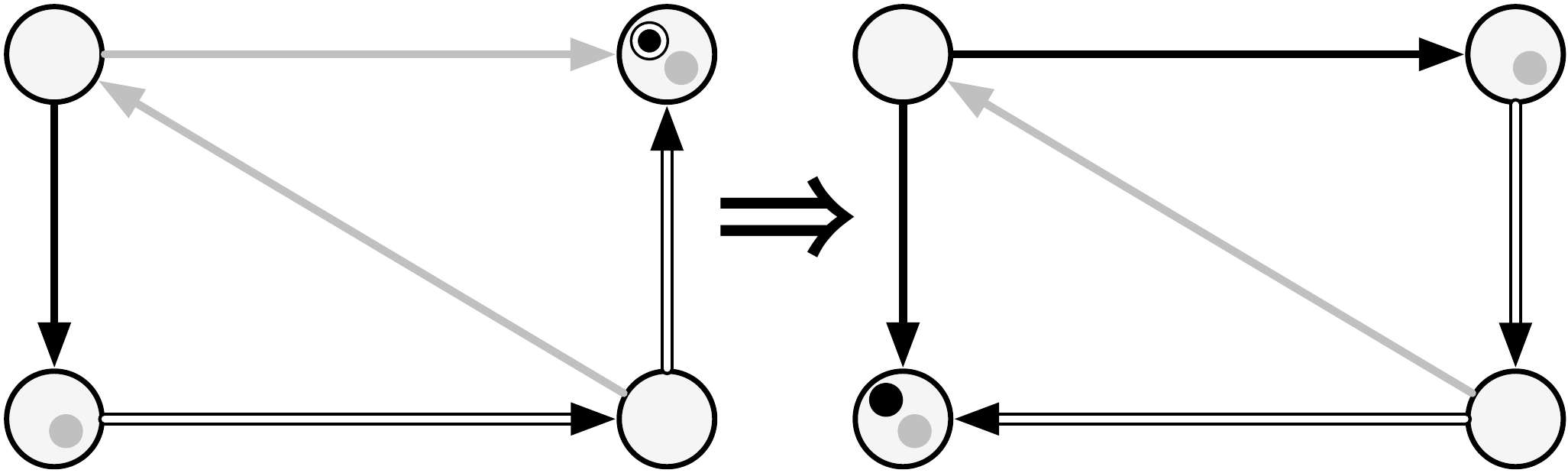} }
	
	\caption{Eliminating {\bf (M2)} moves: (a) the first step to move the black pebble along the 
	doubled path is {\bf (M2)}; (b) avoiding the 
	{\bf (M2)} and simplifying the path.}
	\label{fig.m2-move-eliminate-2}
	\end{figure}
	
	\begin{proof}
	Without loss of generality, we can assume that our sequence of moves reorients a 
	simple path in $H$, and that the first move (the end of the path)
	is {\bf (M2)}.  The {\bf (M2)} step moves a pebble 
	of color $c_i$ from a vertex $w$ onto the edge $vw$, which is reversed.  Because the move
	is {\bf (M2)}, $v$ and $w$ are contained in a maximal monochromatic tree of color $c_i$.
	Call this tree $H'_i$, and observe that it is rooted at $w$.
	
	Now consider the edges reversed in our sequence of moves.  As noted above, before we make
	any of the moves, these sketch out a simple path in $H$ ending at $w$.  Let $z$ be the 
	first vertex on this path in $H'_i$.  We modify our sequence of moves as follows: delete,
	from the beginning, every move before the one that reverses some edge $yz$; prepend onto
	what is left a sequence of moves that moves the pebble on $w$ to $z$ in $H'_i$.  
	
	Since no edges change color in the beginning of the new sequence, we have eliminated
	the {\bf (M2)} move.  Because our construction does not change 
	any of the edges involved in the remaining tail of the original sequence, the 
	part of the original path that is left in the new sequence will still be 
	a simple path in $H$, meeting our initial hypothesis.
	
	The rest of the lemma follows by induction.
	 \end{proof}
	
	Together \reflem{can-kill-m1-moves} and \reflem{kill-m2-moves-locally} prove the following. 
	\begin{lemma}
		If $G$ is a pebble-game graph, then $G$ has a canonical pebble game construction. 
		\labellem{canonical-constructions-exist} 
	\end{lemma}
	
	Using canonical pebble game constructions, we can identify 
	the tight pebble-game graphs with maps-and-trees and \ellteekay\, graphs. 
	
	\begin{restate}{canonical-decomposition-I}[{\bf Main Theorem (Lower Range): Maps-and-trees coincide with pebble-game graphs}]
	Let $0\le \ell\le k$. A graph $G$ is a tight pebble-game graph if and only if $G$ is a $(k,\ell)$-maps-and-trees.
	\end{restate}
	\begin{proof}
		
		As observed above, a maps-and-trees decomposition is a special case of the pebble 
		game decomposition.  Applying \refthm{non-canonical-decomposition}, we see that
		any maps-and-trees must be a pebble-game graph.
		
		For the reverse direction, consider a canonical pebble game construction of a tight graph. 
		From \reflem{pebble-graphs-are-sparse}, we see that there are $\ell$ pebbles left on 
		$G$ at the end of the construction.
		The definition of the {\bf canonical add-edge} move 
		implies that there must be at least one pebble of 
		each $c_i$ for $i=1,2,\ldots,\ell$.  It follows that there is exactly one of each of these 
		colors.  By \reflem{roots}, each of these pebbles is the root of a monochromatic tree-piece
		with $n-1$ edges, yielding the required $\ell$ edge-disjoint spanning trees.
	 \end{proof}
	\begin{restatecor}{m-a-t-equals-tight}
		[\textbf{Nash-Williams \cite{nash-williams:decomposition-into-forests:1964}, Tutte \cite{tutte:decomposing-graph-in-factors-1961}, White and Whiteley \cite{whiteley:union-matroids}}] 
		Let $\ell\le k$. A graph $G$ is tight if and only if has a $(k,\ell)$-maps-and-trees decomposition.
	\end{restatecor}
	
	We next consider the decompositions induced by canonical pebble game 
	constructions when $\ell\ge k+1$. 
	\begin{restate}{canonical-decomposition-II}[{\bf Main Theorem (Upper Range): Proper Trees-and-trees coincide with pebble-game graphs}]
		Let $k\le \ell\le 2k-1$. A graph $G$ is a tight pebble-game graph if and only if it is a proper \ellteekay with $kn-\ell$ edges. 
	\end{restate}
	\begin{proof}
		As observed above, a proper \ellteekay decomposition must be sparse.  
		What we need to show is that a canonical pebble game construction of a tight 
		graph produces a proper \ellteekay.
		
		By \refthm{non-canonical-decomposition} and \reflem{canonical-constructions-exist}, 
		we already have the condition on tree-pieces and the decomposition into $\ell$ edge-disjoint
		trees.  Finally, an application of {\bf (I4)}, shows that every vertex must in in 
		exactly $k$ of the trees, as required.
	 \end{proof}
	
	\begin{restatecor}{t-a-t-equals-tight}[\textbf{Crapo \cite{crapo:rigidity:88}, Haas \cite{haas:2002}}] 
		Let $k\le \ell\le 2k-1$. A graph $G$ is tight if and only if it is a proper \ellteekay.
	\end{restatecor}
	
	\section{Pebble game algorithms for finding decompositions}
	A naïve implementation of the constructions in the previous section leads to an 
	algorithm requiring $\Theta(n^2)$ time to collect each pebble in a canonical 
	construction: in the worst case $\Theta(n)$ applications of the construction 
	in \reflem{kill-m2-moves-locally} requiring $\Theta(n)$ time each, giving a 
	total running time of $\Theta(n^3)$ for the \textbf{decomposition} problem.
	
	In this section, we describe algorithms for the {\bf decomposition} problem
	that run in time $O(n^2)$.  We begin with the overall structure of the algorithm.

	\begin{algorithm}
		[\textbf{The canonical pebble game with colors}]\labelalg{canonical-pebble-game}
\qquad 

\noindent		{\bf Input:} A graph $G$. \\
		{\bf Output:} A pebble-game graph $H$. \\
		{\bf Method:} 
		\begin{itemize}
			\item Set $V(H)=V(G)$ and place one pebble of each color on the vertices of $H$.
			\item 	For each edge $vw\in E(G)$ try to collect at least $\ell+1$ pebbles on $v$ and $w$ using {\bf pebble-slide} moves as described by \reflem{kill-m2-moves-locally}. 
			\item If at least $\ell+1$ pebbles can be collected, add $vw$ to $H$ using an {\bf add-edge} move as in \reflem{can-kill-m1-moves}, otherwise discard $vw$. 

			\item Finally, return $H$, and the locations of the pebbles.
		\end{itemize}
	\end{algorithm}
	
	\paragraph{Correctness.}  \refthm{sparse-graphs-are-pebble-graphs} and the
	result from \cite{whiteley:union-matroids} that the sparse graphs are the independent
	sets of a matroid show that $H$ is a maximum sized sparse subgraph of $G$.  Since
	the construction found is canonical, the main theorem shows that the coloring of 
	the edges in $H$ gives a maps-and-trees or proper \ellteekay decomposition.

	\paragraph{Complexity.} We start by observing that the running time of \refalg{canonical-pebble-game} is the time taken to process $O(n)$ edges added to $H$ and $O(m)$ edges not added to $H$. We first consider the cost of an edge of $G$ that is added to $H$.

	Each of the pebble game moves can be implemented in constant time. 
	What remains is to describe an efficient way to find and move the pebbles.  We use the 
	following algorithm as a subroutine of \refalg{canonical-pebble-game} to do this.
	
	\begin{algorithm}
		[\textbf{Finding a canonical path to a pebble.}]\labelalg{find-one-pebble}
\qquad \\
\noindent{\bf Input:} Vertices $v$ and $w$, and a pebble game configuration on a directed graph $H$. \\
\noindent {\bf Output:} If a pebble was found, `yes', and `no' otherwise. The configuration of $H$ is updated. \\
\noindent		{\bf Method:}
\begin{itemize}
	\item	Start by doing a depth-first search from from $v$ in $H$. If no pebble not on $w$ is found, stop and return `no.'

	\item	Otherwise a pebble was found.  We now have a path $v=v_{1},e_{1},\ldots,e_{p-1},v_{p}=u$, where the $v_{i}$ are vertices and $e_{i}$ is the edge $v_{i}v_{i+1}$. Let $c[e_{i}]$ be the color of the pebble on $e_{i}$.  We will use the array $c[]$ to keep track of the colors of pebbles on vertices and edges 
		after we move them and the array $s[]$ to sketch out a canonical path from 
		$v$ to $u$ by finding a successor for each edge.

	\item Set\, $s[u]=`end'$ and set\, $c[u]$ to the color of an arbitrary pebble on $u$. 
	We walk on the path in reverse order: $v_{p},e_{p-1}, e_{p-2}, \ldots, e_{1},v_{1}$.
	For each $i$, check to see if $c[v_{i}]$ is set; if so, go on to the next $i$. Otherwise, check to see if $c[v_{i+1}]=c[e_{i}]$. 
	\item If it is, set $s[v_{i}]=e_{i}$ and set $c[v_{i}]=c[e_{i}]$, and go on to the next edge.
	\item	Otherwise $c[v_{i+1}]\neq c[e_{i}]$, try to find a monochromatic path in color $c[v_{i+1}]$ from $v_{i}$ to $v_{i+1}$. If a vertex $x$ is encountered for which $c[x]$ is set, we have a path $v_{i}=x_{1},f_{1},x_{2},\ldots,f_{q-1},x_{q}=x$ that is monochromatic in the color 
	of the edges; set $c[x_{i}]=c[f_{i}]$ and $s[x_{i}]=f_{i}$ for $i=1,2,\ldots,q-1$. If $c[x]=c[f_{q-1}]$, stop. Otherwise, recursively check that there is not a monochromatic $c[x]$ path from $x_{q-1}$ to $x$ using this same procedure.

	\item	Finally, slide pebbles along the path from 
	the original endpoints $v$ to $u$ specified by the successor array $s[v]$, $s[s[v]]$, $\ldots$
	\end{itemize}
	\end{algorithm}

	The correctness of \refalg{find-one-pebble} comes from the fact that it is
	implementing the shortcut construction.
	Efficiency comes from the fact that instead of potentially moving the 
	pebble back and forth, \refalg{find-one-pebble} pre-computes a canonical
	path crossing each edge of $H$ at most three times: once in the initial
	depth-first search, and twice while converting the initial path to a 
	canonical one.  It follows that each accepted edges takes $O(n)$ time, for
	a total of $O(n^2)$ time spent processing edges in $H$.

	Although we have not discussed this explicity, for the algorithm to be efficient
	we need to maintain components as in \cite{pebblegame}.
	After each accepted edge, the components of $H$ 
	can be updated in time $O(n)$. Finally,  the results of 
	\cite{pebblegame,cccg} show that the rejected edges take an amortized $O(1)$ time each.

	Summarizing, we have shown that the canonical pebble game with colors 
	solves the {\bf decomposition} problem in time $O(n^2)$.
	
	\section{An important special case: Rigidity in dimension $2$ and slider-pinning}
	In this short section we present a new application for the special case of 
	practical importance, $k=2$, $\ell=3$.  As discussed in the introduction,
	Laman's theorem \cite{laman} characterizes minimally rigid graphs as the $(2,3)$-tight
	graphs.  In  recent work on slider pinning,  developed after the current 
	paper was submitted, we introduced the slider-pinning model of rigidity \cite{sliders,genericity}.
	Combinatorially, we model the bar-slider frameworks as simple graphs together with some 
	loops placed on their vertices in such a way that there are no more than $2$ loops per 
	vertex, one of each color.  
	
	We characterize the minimally rigid bar-slider graphs \cite{genericity}
	as graphs that are: 
	\begin{enumerate}
		\item $(2,3)$-sparse for subgraphs containing no loops.
		\item $(2,0)$-tight when loops are included.
	\end{enumerate}
	We call these graphs $(2,0,3)$-graded-tight, and they are a special case of 
	the graded-sparse graphs studied in our paper \cite{graded}.

	The connection with the pebble games in this paper is the following.
	\begin{corollary}[{\bf Pebble games and slider-pinning}]
		In any $(2,3)$-pebble game graph, if we replace pebbles by loops, 
		we obtain a  $(2,0,3)$-graded-tight graph.
	\end{corollary}
	\begin{proof}
		Follows from invariant {\bf (I3)} of \reflem{pebble-game-invariants}.
	\end{proof}
	
	In \cite{sliders}, we study a special case of slider
	pinning where every slider is either vertical or horizontal.  We model the sliders 
	as pre-colored loops, with the color indicating $x$ or
	$y$ direction.  For this axis parallel slider case, the minimally rigid graphs are 
	characterized by:
	\begin{enumerate}
		\item $(2,3)$-sparse for subgraphs containing no loops.
		\item Admit a $2$-coloring of the edges so that each color is a forest
		(i.e., has no cycles),
		and each monochromatic tree spans exactly one loop of its color.
	\end{enumerate}
	
	This also has an interpretation in terms of colored pebble games.
	\begin{corollary}[{\bf The pebble game with colors and slider-pinning}]
		In any canonical $(2,3)$-pebble-game-with-colors graph, if we replace pebbles by loops
		of the same color, we obtain the graph of a minimally pinned axis-parallel bar-slider 
		framework.
	\end{corollary}
	\begin{proof}
		Follows from \refthm{canonical-decomposition-II}, and \reflem{roots}.
	\end{proof}

	\section{Conclusions and open problems}
	We presented a new characterization of $(k,\ell)$-sparse graphs, the 
	\textbf{pebble game with colors}, and
	used it to give an efficient algorithm for finding decompositions of sparse 
	graphs into edge-disjoint trees.  Our algorithm finds such sparsity-certifying 
	decompositions in the upper range and 
	runs in time $O(n^2)$, which 
	is as fast as the algorithms for recognizing sparse graphs in the upper range 
	from \cite{pebblegame}.

	We also used the pebble game with colors to describe a new sparsity-certifying 
	decomposition that applies to the entire matroidal range of sparse graphs.

	We defined and studied a class of canonical pebble game constructions that correspond to either a
	 maps-and-trees or proper \ellteekay decomposition. This gives a new proof of the Tutte-Nash-Williams 
	arboricity theorem and a unified proof of the previously studied decomposition certificates of sparsity.
	Canonical pebble game constructions also show the relationship
	between the $\ell+1$ pebble condition, which applies to the upper range of 
	$\ell$, to matroid union augmenting paths, which do not apply in the upper range.

	\paragraph{Algorithmic consequences and open problems.}
	In \cite{gabow:forests:1992}, Gabow and Westermann give an $O(n^{3/2})$ algorithm for recognizing sparse
	graphs in the lower range and extracting sparse subgraphs from dense ones.  Their technique 
	is based on efficiently finding matroid union augmenting paths, which extend  a maps-and-trees
	decomposition.  The $O(n^{3/2})$ algorithm uses two subroutines to find augmenting paths: {\bf 
	cyclic scanning}, which finds augmenting paths one at a time, and {\bf batch scanning}, which 
	finds groups of disjoint augmenting paths.
	
	We observe that \refalg{canonical-pebble-game} can be used to replace cyclic scanning in 
	Gabow and Westermann's algorithm without changing the running time.  The data structures 
	used in the implementation of the pebble game, detailed in \cite{cccg,pebblegame} are 
	simpler and easier to implement than those used to support cyclic scanning.
	
	The two major open algorithmic problems related to the pebble game are then:
	
	\begin{problem}
		Develop a pebble game algorithm with the properties of {\bf batch scanning} and
		obtain an implementable $O(n^{3/2})$ algorithm for the lower range.
	\end{problem}
	
	\begin{problem}
		Extend {\bf batch scanning} to the $\ell+1$ pebble condition and derive an $O(n^{3/2})$
		pebble game algorithm for the upper range.
	\end{problem}
	In particular, it would be of practical importance 
	to find an implementable $O(n^{3/2})$ algorithm for decompositions 
	into edge-disjoint spanning trees.

\end{document}